\documentclass[A4,12pt]{amsart}
\usepackage[cp1250]{inputenc}
\usepackage[OT2,OT4]{fontenc}
\usepackage[T1]{fontenc}
\usepackage{amssymb}

\usepackage[T1]{fontenc}
\usepackage{amssymb}
\newcommand{\cf}{\operatorname{cf}}

\newcommand{\Gee }{\mathcal G}

\newcommand{\Ree }{\mathcal R}
\newcommand{\Bee }{\mathcal B}

\newcommand{\Hee }{\mathcal H}
\newcommand{\Wee }{\mathcal W}
\newcommand{\Qee }{\mathcal Q}

\newcommand{\Pee }{\mathcal P}
\newcommand{\Fee }{\mathcal F}

\newcommand{\Vee }{\mathcal V}

\renewcommand{\int}{\operatorname{Int}}

\newtheorem{theorem}{Theorem}
\newtheorem{corollary}[theorem]{Corollary}
\newtheorem{lemma}[theorem]{Lemma}

\author{Piotr Kalemba}
\address{Piotr Kalemba \\
 Institute of Mathematics, University of
Silesia \\
ul. Bankowa 14, 40-007 Katowice}
\email{pkalemba@math.us.edu.pl}
\author{Szymon Plewik}
\address{Szymon Plewik\\Institute of Mathematics,
University of Silesia, ul. Ban\-ko\-wa 14, 40-007 Katowice}
\email{plewik@math.us.edu.pl}
\begin{document}

\title{Ideals which generalize $(v^0)$}
%\date{}
\maketitle 

\begin{abstract} We consider ideals $d^0(\Vee)$ which are generalizations of the ideal $(v^0)$. We formulate  couterparts of  Hadamard's theorem. Then,  adopting the base tree theorem and applying  Kulpa-Szymański Theorem, we obtain  $ cov(d^0(\Vee))\leq add(d^0(\Vee))^+$.
\end{abstract}

  \section{Introduction}  
  Let $\omega$ denotes the set of all natural numbers and  $|X|$ denotes the cardinality of a set $X$. If $\Fee$ is a family of sets, then $add(\Fee)$ is the least cardinality of a subfamily $\Gee \subseteq \Fee$ such that the union $\bigcup \Gee$ is not in $\Fee$, but  $cov(\Fee)$ is the least cardinality of a subfamily $\Gee \subseteq \Fee$ such that $\bigcup \Gee = \bigcup \Fee$. Also,  $d^0(\Fee)$ denotes an ideal defined as  follows. A set $X\in d^0(\Fee)$, whenever $X\subset \bigcup \Fee$ and for each $V\in \Fee$ there exists $U \in \Fee$ such that $U \subseteq V$ and $U \cap X =\emptyset$.

 In this note, products $\prod X_i$ are examined  where  $i \in \omega$. In fact, it is assumed that each $X_i$ is a  finite  discrete space  with more than one point.  For infinite $X_i$, we
 make some  comments, only.   Considerations are related to  some trees  which we call  trimmed trees. Each trimmed tree $T[A,\alpha]$ is uniquely determined by two parameters  $A \in [\omega]^\omega$ and $\alpha \in \prod X_i$, and it is a subset of the union $\bigcup \{ X_0\times X_1\times \ldots \times X_n: n\in \omega \}=  \prod_{fin}X_i$.  As usual,  elements of a trimmed tree  $T[A,\alpha] \subseteq \prod_{fin}
X_i$ are called nodes. Unions of nodes which belong to $\prod X_i$ are called branches. The family of all branches of a trimmed tree is  the  perfect subset $[T[A,\alpha ]]\subseteq \prod X_i$. So, our terminology is standard, compare  \cite{jec} or \cite{ker}. 

 We are guided by the  notion $J_f(\alpha)$ which was considered by G. Moran and D. Strauss \cite{ms}.  For a given function $f$ defined on $\prod X_i$ and a fixed point $\alpha\in\prod X_i$ the family  $J_f(\alpha) \subseteq [\omega]^\omega$ consists of all sets $A\in[\omega]^\omega$ such that $f$ is constant on the set of all functions $\beta\in\prod X_i$ such that $\alpha(n)=\beta(n)$ for $n\in \omega \setminus A$. There are conditions sufficient for   some $J_f(\alpha)$ to be non-empty, see \cite{ms} Theorems 1 - 3. In fact, $A\in J_f(\alpha)$ says that the function $f$ is constant on the perfect set $[T[A,\alpha ]]$.  
  
  In \cite{bre}, \cite{kpw} or  \cite{knw}  were considered  properties of the ideal $(v^0)$.   Ideals  $d^0(\{[T]: T\in \Vee\})$ are generalizations of $(v^0)$. In fact,   $(v^0)$  is $d^0(\{[T]: T\in \Vee\})$, where it is  assumed that $X_n =\{0,1\}$ for all $n\in \omega$. Here, we generalize  properties of the ideal $(v^0)$ for  sets $X_n$ being  finite and with more than one point. Albeit, we do not consider properties which depend on the assumption that sets $X_n$ have less than $k$-points for a fixed $k\in \omega$, compare Theorem 3 in \cite{ms}, or that the set  $\{n: |X_n|=k\}$ is finite for each $k\in \omega$, compare Example 1 in \cite{ms}.   We  adopt the base tree theorem, see \cite{bps}  and \cite{bs}.  Applying so called Kulpa-Szymański Theorem, we obtain the final result of this paper, i.e. inequalities  $add(d^0(\Vee))\leq cov(d^0(\Vee))\leq add(d^0(\Vee))^+$ under the hypothesis that  factors $X_i$ of a product $\prod X_i$  are finite and  have more than one point.

\section{Trimmed trees}
Suppose  $X_0, X_1, \ldots$ be an infinite sequence of sets with more than one point. Let $\prod X_i$ denotes the Cartesian product of theses sets. Fix a function $\alpha \in \prod X_i$ and an infinite subset $A \subseteq \omega$.
The subset  $$\{ \beta \in \prod X_i: \alpha|_{\omega \setminus A}= \beta|_{\omega\setminus A} \}
$$ is a perfect subset of $\prod X_i$  equipped with product topology, where each $X_n$ is considered with the  discrete topology. We will denote this subset $[T[A, \alpha]]$, since it can be described as the set of all infinite branches of the tree which we call a \textit{trimmed} tree.  

Any trimmed tree $T[A,\alpha ] \subseteq \prod_{fin}X_i$ is the union of sets of nodes $T[A,\alpha, n]$, where $ n \in \omega$. We assume that the empty set is a node of any $T[A,\alpha]$, too. One  obtains sets of nodes $T[A,\alpha, n]$ by the following  procedure   $T[A,\alpha,n]= \{s \cup  \{(n,\alpha(n))\}: s \in  T[A,\alpha, n-1]\}$, whenever  $n\notin  A$, but   $T[A,\alpha, n]= \{s \cup  \{(n, x)\}: s \in  T[A,\alpha, n-1] \mbox { and } x \in X_n\}$ for $n\in A$. This procedure is inductive, so we assume that  $T[A,\alpha,0] = \{\{(0,\alpha(0))\}\}$, whenever  $0\notin  A$,  and  $T[A,\alpha, 0] = \{\{(0,x)\}: x\in X_0\}$, for  $0\in A$. 
In consequence $T[A,\alpha] = \bigcup \{ T[A,\alpha, n]: n \in \omega\} \cup \{\emptyset \}$ and $
[T[A, \alpha]] = \{ \beta \in \prod X_i: \alpha|_{\omega \setminus A}= \beta|_{\omega\setminus A} \}.
$

  Obviously, $T[A, \alpha]$ does not depend on the restriction $\alpha|_A$. Also,   $T[A, \alpha] = T[B, \beta]$ implies $\alpha|_{\omega \setminus A} = \beta|_{\omega \setminus B}$, what means $A=B$ and $\alpha (n) = \beta (n)$ for  $n \in \omega\setminus A$.

Suppose $\prod X_i$ is fixed. Put   
$$\Vee = \{T[A,\alpha]: A \in [\omega]^\omega \mbox{ and } \alpha \in  \prod   X_i   \}.$$  If $\alpha\in \prod X_i $ and $s\in \prod_{fin} X_i$, then let $\alpha_s$ denotes the unique function $\beta \in \prod X_i$ such that  $s\subset \beta \subseteq \alpha \cup s$. If  $Y\subseteq \prod X_i$ and  $s\in \prod_{fin} X_i$, then $Y_s$ denotes the  set $\{\alpha_{s} : \alpha\in Y\}$, but $Y^*$ denotes the family of all functions $ \alpha_s$, where $ \alpha\in Y$ and $ s\in \prod_{fin} X_i $.   Thus $$Y^*=\bigcup \{Y_s: s \in \prod_{fin} X_i\}.$$ Finally,  put 
$\Vee^* = \{[T]^*: T \in  \Vee  \}$. From now on, we assume that a Cartesian product $\prod X_i$ is fixed where   each  $X_i$ is a set with more than one point.  This  will be  applied to symbols $\Vee$ or $\Vee^*$, always. A few facts will need  that each  set $X_i$ is finite, additionally.

\begin{lemma}\label{continuum} 
Each member  of $\Vee^*$ contains continuum many pairwise disjoint members of $\Vee^*$.
 \end{lemma} 
\begin{proof} 
Fix $[T]^*\in\Vee^*$, where $T=T[A,\alpha]$.  Let $\Ree(A)$ be an almost disjoint family of the cardinality continuum consisting of infinite subsets of $A$. Fix a function $ \gamma\in\prod X_i$ such that $\gamma(n)\neq\alpha(n)$, for any $n\in A$.  For each $C\in\Ree(A)$ choose an infinite  subset $V_C\subset C$ such that $C\setminus V_C$ is infinite, too. Let $\alpha_C\in\prod X_i$ be a function such that \[ \alpha_C(n)=
\begin{cases}
\gamma(n), & \mbox{ for }n\in C;\\
\alpha(n),  & \mbox{ for }n\notin C.
\end{cases} \] The family $\{[T[V_C,\alpha_C]]^*: C\in\Ree(A)\}$ is a desired one.
\end{proof}

  If  $T\in\Vee $ and  $s\in \prod_{fin} X_i$, then    $T_s$ denotes the tree $\{\alpha_{s}|n : \alpha\in [T] \mbox{ and } n\in \omega \}$. Note that, notions $T_s$ and  $Y_s$ are used in different contexts. Each tree $T_s$ consists of nodes, but any $Y_s$ consists of infinite sequences. 
 We have assumed that each  $X_n$ has at least two points, hence any tree $ T[A,\alpha]$ has continuum many branches. If $T=T[A,\alpha]$ and $s \in \prod_{fin} X_i$, then $T_s=T[A\setminus |s|, \alpha_s]$. Therefore,  $T\in\Vee $ implies  $T_s\in\Vee $. 

\section{Fusion relations}

Let  $a_0, a_1, \ldots $ and $b_0, b_1, \ldots $ be increasing enumerations of all points of $A=\{a_n: n \in \omega\}\in [\omega ]^\omega$ and $B\in [\omega ]^\omega$, respectively.  Put  $$T[A,\alpha] \subseteq_n T[B,\beta],$$ whenever  $T[A,\alpha] \subseteq T[B,\beta ]$ and  $a_0=b_0$, $a_1=b_1$, $\ldots , a_n=b_n$. Thus, the decreasing  sequence of relations  $(\subseteq_n)_{n\in\omega}$ is defined.  These relations hold between elements of $\Vee$. Always, $\subseteq_{n+1}$ is contained in $ \subseteq_n$.  So, we can apply  the  method of fusion, compare \cite{jec}, using these relation  to trimmed trees. In many papers,  facts about fusion are presented without proof.  Since  details considered here are not so obvious, we run  full proofs here.  

\begin{lemma}\label{fuzja2}
Let   $(T[A_n, \alpha_n])_{n\in\omega}$ be a sequence of elements of $\Vee$.  If always  $T[A_{n+1}, \alpha_{n+1}]\subseteq_n T[A_n, \alpha_n]$, then there exists $C\in [\omega]^\omega$ and $\alpha \in \prod X_i$ such that  $T[C, \alpha]\subseteq T[A_n, \alpha_n]$ for any $n\in\omega$.
\end{lemma} 
\begin{proof}       An inclusion  $T[A_{n+1}, \alpha_{n+1}]\subseteq T[A_n, \alpha_n]$ implies $A_{n+1}\subseteq A_n$ and $ \alpha_{n}|_{\omega\setminus A_n}=\alpha_{n+1}|_{\omega\setminus A_n}$ and  $ \alpha_{n}|_{\omega\setminus A_n}=\alpha_{n+1}|_{\omega\setminus A_n}\subseteq \alpha_{n+1}|_{\omega\setminus A_{n+1}}.$ Hence,  the union $\bigcup\{\alpha_n|_{\omega\setminus A_n}: n\in\omega \}$ is a function. Fix $\alpha \in \prod X_i$ which extends this union. Functions $\alpha$ and $\alpha_n$ coincide on the set  $\omega \setminus A_n$, thus  $T[A_{n}, \alpha] = T[A_n, \alpha_n]$.
Let   $a_0^n, a_1^n, \ldots$ be the increasing enumeration of points of $A_n$. Put $C=\{a_n^n: n\in\omega\}$.   If $k\geq n$, then $a^k_k\in A_k\subseteq A_n$. If $k<n$, then   $$ T[A_n, \alpha_n]= T[A_n,\alpha]\subseteq_{n-1} T[A_{n-1},\alpha]\subseteq_{n-2} \ldots  \subseteq_{k}T[A_k,\alpha], $$ hence $ T[A_n, \alpha] \subseteq_{k}T[A_k,\alpha]. $ This implies   $a^k_k\in A_n$. Therefore  $C\subseteq A_n$.  Finally,    $T[C,\alpha]\subseteq T[A_n,\alpha]=T[A_n,\alpha_n]$ for any  $n\in\omega$. 
\end{proof}

From now on, the ideal $d^0(\{[T]: T\in \Vee\})$ will be shortly denoted $d^0(\Vee)$. 

\begin{lemma} \label{5} Let  $s, t\in \prod_{fin} X_i$ and $Y\subseteq \prod X_i$ and assume that  $|X_k|< add (d^0(\Vee))$ for any $k\in \omega$. Then,  $Y_s\in d^0(\Vee)$ if and only, if $Y_t\in d^0(\Vee)$.
\end{lemma}
\begin{proof}  Suppose  $|s|=|t|$. Fix $T\in\Vee$. Let  $Y_s\in d^0(\Vee)$. If $t\notin T$, then $[T]\cap Y_t=\emptyset$. If $t\in T$, then choose $P_s\in \Vee$ such that  $[P_s]\cap Y_s=\emptyset$ and $P_s\subseteq T_s$. Hence, $P_t\subseteq T_t\subseteq T$ and $[P_t]\cap Y_t=\emptyset$. Therefore $Y_t\in d^0(\Vee)$.

 Suppose $s\in \prod_{fin} X_i$ and  $Y_s \in  d^0(\Vee)$. Since  $Y$ is a subset of the union $\bigcup\{Y_u: u\in \prod_{fin} X_i\mbox{ and } |u|=|s|\}, $  then it is contained in an union of less than $add (d^0(\Vee))$ many elements  of  $d^0(
\Vee )$. Hence  $Y\in d^0(\Vee )$. 

 Now, suppose that $Y \in d^0(\Vee)$ and  $s\in \prod_{fin} X_i$. If  $t \in \prod_{fin} X_i$  and $|s|=|t|$, then  $\{\alpha\in Y: \alpha|_{|s|}=t\} =Y_t\cap Y=(Y_t\cap Y)_t\subseteq Y.$ Thus  $(Y_t\cap Y)_s \in d^0(\Vee)$.  Also,   $Y_s=\bigcup\{(Y_t\cap Y)_s: t\in \prod_{fin} X_i \mbox{ and }|s|=|t|\}$. Hence $Y_s$ is an union of less than $add (d^0(\Vee))$ many elements  of  $d^0(
\Vee )$. Therefore $Y_s \in d^0(\Vee )$. 
\end{proof}

\begin{lemma}\label{qqq} Suppose that $|X_i| < add (d^0(\Vee))$ for all $i\in \omega$.  Let $k\in\omega$ and $T\in\Vee$. If $Y\in d^0(\Vee)$, then there exists a tree  $P\in\Vee$ such that 
$P\subseteq_k T$ and $[P]\cap Y=\emptyset$.
\end{lemma}
\begin{proof} Let  $T=T[A,\alpha]$ and $a_0, a_1,\ldots$ be the increasing enumeration of all points of $A$.  Consider the union 
 $U = \bigcup\{Y_s: s\in X_0\times\ldots\times X_{a_k}\}$. It consists of less than $add (d^0(\Vee))$ many sets, each one from  $d^0(\Vee)$ by Lemma \ref{5}. Thus $U \in d^0(\Vee)$.  Take $Q\in\Vee$ such that  $Q\subseteq T$ and   $[Q]\cap U=\emptyset$. Check that $[Q_s]\cap Y=\emptyset$ for any $s\in X_0\times\ldots\times X_{a_k}$. Finally put   $P=\bigcup \{Q_{s}: s \in T \cap X_0\times\ldots\times X_{a_k}\}$. 
\end{proof}

 We do not know, whether lemmas \ref{fuzja2}, \ref{5} and \ref{qqq} are valid for $\prod X_i = \omega^\omega$. 
  Their proofs would work if $d^0(\Vee)$ was a $\sigma$-ideal.  If all sets $X_i$ are finite, then hypotheses of  these lemmas are fulfilled. Moreover, then $add (d^0(\Vee))$ is an uncountable cardinal.

\begin{theorem}\label{tss}
If each $X_i$ is a finite set, then   $d^0(\Vee)$ is a $\sigma$-ideal. 
\end{theorem}
\begin{proof} Assume $S_0, S_1,\ldots$ is an increasing sequence of elements of the ideal $d^0(\Vee)$.  Fix $T\in\Vee$ and put $T_0=T$. Using Lemma \ref{qqq},  choose inductively  trees $T_k\in\Vee$ such that $T_{k+1}\subseteq_{k}T_{k}$ and $[T_{k}]\cap S_k=\emptyset$.  Thus, it has been defined a sequence of elements of $\Vee$ satisfying hypotheses of  Lemma \ref{fuzja2}. So, there exists a tree $P\in\Vee$ such that $P\subseteq T$ and $[P]\cap \bigcup_{k\in\omega}S_k=\emptyset$. 
\end{proof}

\begin{corollary}\label{gwiazdka} If each $X_i$ is a finite set and $Y\in d^0(\Vee)$,  then  $Y^*\in d^0(\Vee)$.
\end{corollary}
\begin{proof} By Lemma \ref{5}, $Y_s \in d^0(\Vee)$ for each $s\in \prod_{fin} X_i$. Since  $Y^*= \bigcup\{Y_s: s\in \prod_{fin} X_i \}$, it is a countable union of elements of $d^0(\Vee)$. Thus    $Y^*\in d^0(\Vee)$ by   Theorem \ref{tss}.
\end{proof}

\section{Counterparts of Hadamard's theorem}

Two  sequences of countable sets $(a_n)_{n\in \omega}$ and  $(b_n)_{n\in \omega}$ form a $(\omega,\omega)$-gap, whenever  $$a_0 \subset^* a_{1} \subset^* \ldots a_n \subset^*\ldots  \subset^* b_{n} \subset^*\ldots \subset^* b_1\subset^* b_0$$ and no set $c$
fulfills $ a_n \subset^* c \subset^* b_{n} $ for all $n \in \omega$. The famous Hadamard's theorem says that there are no  $(\omega,\omega)$-gaps, compare \cite{had} or \cite{sch}.  This theorem can be formulated in our's notations. Indeed, assume that always $X_i = \{0,1\}$ and identify each subset $Y\subseteq  \omega$ with its characteristic function which belongs to $\prod X_i$. Then one can check that Hadamard's theorem is equivalent to the property that any decreasing  sequence of  elements of $\Vee^*$ has a lower bound. Theorem \ref{*fuzja} extends this property.

If  $T = T[A,\alpha] \in \Vee$, then the tree $T$ is determined by the function   
\[ \delta(T[A,\alpha])(n) = \begin{cases}\{\alpha(n)\},& \mbox{ whenever }  n\notin A;\\
X_n,& \mbox{ whenever } n\in A. \end{cases} \] 
For a tree $P\in\Vee$, we have $\beta \in[P]$  if and only, if $  \beta(n) \in \delta(P)(n)$ for each $n\in \omega$.  Also, $\beta\in [P]^*$ if and only, if $  \beta(n) \in \delta(P)(n)$ for all, but finitely many $n\in \omega$. 
One can check that, if $[P]^*, [T]^*\in\Vee^*$, then $[P]^*\subseteq [T]^*$ if and only, if  $\delta(P)(k)\subseteq\delta(T)(k)$ for all,  but finitely many $k\in \omega$.    
 
\begin{lemma}\label{fakt}
If $P, T\in \Vee$ and $[P]^*\subseteq [T]^*$, then for each $n\in\omega$ there exists a tree  $Q\in \Vee$ such that $Q\subseteq_n T$ and $[Q]^*=[P]^*$. 
\end{lemma}
\begin{proof}
Fix $k_0\in \omega$ such that $\delta(P)( k)\subseteq \delta(T)( k)$,  for any $k\geq k_0$. If $T=T[A,\alpha]$, then let $a_n$ be the   $n$-th element of  $A$. 
Put \[\delta(Q)( m)= \begin{cases} \delta(T)(m), & \mbox{ for }   m\leq \max\{a_n, k_0\}, \\ \delta(P)(m), & \mbox{ for }  m>\max\{a_n, k_0\}. \end{cases}\] The function $\delta(Q)$ uniquely determines the tree 
 $Q \in \Vee$ which is a desired one.
\end{proof}

\begin{theorem}\label{*fuzja} Let $(W_n)_{n\in \omega}$ be a sequence of  elements of $\Vee^*$. 
If  $ W_{n+1}\subseteq W_n $ for any $n\in \omega$,  then there exists  $W\in\Vee^*$ such that $W\subseteq W_n$, for any $n\in \omega$.
\end{theorem}
\begin{proof} Choose $T_n \in \Vee$ such that  $W_n=[T_n]^*$.  
Inductively,  construct  a sequence of trees $(Q_n)_{n\in\omega}$ such that  $Q_{n+1}\subseteq_n Q_n \in \Vee$  and $[Q_n]^*=[T_n]^*$, using  Lemma \ref{fakt}.  By Lemma \ref{fuzja2}, there exists  $T\in\Vee$ such that  $T \subseteq Q_n$ and   $W= [T]^*\subseteq [Q_n]^* = [T_n]^*=W_n$,  for all  $n\in \omega$.
\end{proof}

Note that, Lemma \ref{fakt} and Theorem \ref{*fuzja} do not require that sets $X_i$ are finite. Theorem \ref{*fuzja} immediately follows that  $d^0(\Vee^*)$ is a $\sigma$-ideal.

\begin{corollary}\label{rownosc} If each $X_i$ is a finite set, then $d^0(\Vee)=d^0(\Vee^*).$
\end{corollary}
\begin{proof}
 If $Y\in d^0(\Vee)$ and $[T]^*\in\Vee^*$, then  $Y^*\in d^0(\Vee)$ by Corollary  \ref{gwiazdka}. There exists a tree $P\subseteq T$ such that $[P]\cap Y^*=\emptyset$. Hence $[P]^*\cap Y^*=\emptyset$ and $[P]^*\cap Y=\emptyset$ and finally  $Y\in d^0(\Vee^*)$. 
 
 If  
 $Y\in d^0(\Vee^*)$ and $T\in\Vee$, then there exists $[P]^*\subseteq [T]^*$ such that $[P]^*\cap Y=\emptyset$. Therefore $[P]\cap Y=\emptyset$. By Lemma \ref{fakt} one can assume that $P\subseteq T$,  so  $Y\in d^0(\Vee)$.
\end{proof}

Families   $\Vee$ and $\Vee^*$ are not isomorphic with respect to the inclusion. 
  Any decreasing sequence  of elements of  $\Vee^*$ has a lower bound, see Theorem \ref{*fuzja}. But a sequence of trees  $(T[\omega \setminus n ,\alpha])_{n \in \omega}$ has no  lower bound in $\Vee$.
\section{A version of base tree}
Base Matrix Lemma, see \cite{bps}, or Base Matrix Tree, compare \cite{bs} or \cite{kpw},  are adopted to trimmed trees in this part. We omit some proofs, since they are completely analogical to these which are in \cite{bps}, \cite{bs}, \cite{kpw}, etc. From now on, assme that all $X_i$ are countable.
Elements $[Q]^*, [P]^* \in\Vee^*$ are \textit{incompatible} whenever the intersection $[P]^* \cap [Q]^*$ contains no element of   $\Vee^*$. Thus, if $[Q]^*$ and $[P]^*$ are incompatible, then $[P]^* \cap [Q]^*$ is countable. But, if  $[P]^* \cap [Q]^*$ is uncountable, then $[P]^*\cap [T]^*\in\Vee^*.$ Indeed, 
the intersection  $[P]^* \cap [Q]^*$ is a countable union of closed sets  $[P_s]\cap[Q_t]$, where $s,t \in \prod_{fin} X_i.$  Therefore, some $[P_s]\cap[Q_t]$ has to be uncountable and hence  $P_s \cap Q_t$ is a trimmed tree, moreover $[P_s\cap Q_t]^* = [P]^* \cap[Q]^*$.

 A family  $\Pee $ is called  $v$-\textit{partition}, whenever $\Pee$ is a maximal family, with respect to the inclusion, of pairwise incompatible elements of  $\Vee^*$. Any collection of  $v$-partitions is called  $v$-\textit{matrix}. We say that a $v$-partition $\Pee$ refines a $v$-partition $\Qee$ (briefly $\Pee\prec \Qee$), if  for every $[P]^*\in\Pee$ there exists  $[Q]^*\in\Qee$ such that $[P]^*\subseteq [Q]^*$. A matrix $\Hee$ is called \textit{shattering} if for any $[T]^*\in\Vee^*$ there exists a $v$-partition $\Pee \in \Hee$  such that at least two elements of $\Pee$ are compatible with $[T]^*$. The least cardinality of a shattering matrix we denote  $\kappa(\prod X_i)$.

The poset ($\Vee^*$, $\subseteq$) is separative, i.e.  if $[P]^*$ is not contained in $ [T]^*$, then there exists $[Q]^*\in\Vee^*$ such that $[Q]^*\subseteq [P]^*$ and $[Q]^*$ is incompatible with $[T]^*$. Indeed, if  $[P]^*\setminus [T]^*\neq\emptyset$, then the  set  $Z=\{n\in\omega: \delta(P)(n)\setminus \delta(T)(n)\neq\emptyset\}$ is infinite. Fix a set $N\in[Z]^\omega$ such that there exist infinitely many  $n\in\omega\setminus N$ for, which $\delta(P)(n)=X_n$. Fix also, a function $\alpha$ with domain $N$ such that  $\alpha(n)\in\delta(P)(n)\setminus \delta(T)(n)$ for $n\in N$. Put
\[ \delta(Q)(n)=
\begin{cases}
 \delta(P)(n), & \mbox{ for }n\notin N; \\
 \{\alpha(n)\}, & \mbox{ for }n\in N.
\end{cases}  
\] The function $\delta(Q)$ uniquely determines the tree 
 $Q \in \Vee$ which is a disered one.

For any tree $T\in\Vee$ a poset $$(\{[P]^*\in\Vee^*: [P]^*\subseteq [T]^*\}, \subseteq)$$ is isomorphic with the poset $(\Vee^*, \subseteq)$. Moreover, posets  $$(\{P\in\Vee: P\subseteq T\}, \subseteq) \mbox{  and } (\Vee, \subseteq)$$  are isomorphic, too. Additionally, if  $T=T[A,\alpha]$, then each bijection $\phi:A \to\omega$ determines these isomorphisms.  Indeed, for $T[B,\beta]\subseteq T[A,\alpha]$, put $\Phi(T[B,\beta])=T[\phi(B),\beta\circ\phi^{-1}].$  If $T[C,\gamma]\in \Vee$, then put
 $\Psi(T[C,\gamma])=T[\phi^{-1}(C),\gamma\circ\phi \cup \alpha|_{\omega \setminus A}]\subseteq T[A,\alpha]$.
We have $$\Phi\circ\Psi(T[C,\gamma])=T[C,\gamma] \mbox{ and } \Psi\circ\Phi(T[B,\beta])=T[B,\beta].$$ Thus  $\Phi$ and $\Psi$ are mutually inverse bijections. Moreover, $\Phi$ and $\Psi$ preserve the relation of inclusion between trees, hence they are isomorphisms.
Now, we can define  isomorphism $\Phi^*$ of posets $$(\{[P]^*\in\Vee^*: [P]^*\subseteq [T]^*\}, \subseteq) \mbox{  and } (\Vee^*, \subseteq)$$  as follows $\Phi^*([P]^*)=[\Phi(P)]^*$,  for any $[P]^*\subseteq [T]^*$  and $[P]^*\in\Vee^*.$

The next lemma is a counterpart of the lemma  2.6 in   \cite{bps}.
\begin{lemma}\label{EEE} If $\Hee$ is a $v$-matrix of the cardinality less than $\kappa (\prod X_i)$, then there exists a  $v$-partition $\Pee$ which refines each $v$-partition  $\Qee \in \Hee$. 
\end{lemma} 
\begin{proof} Orders $(\{[Q]^*\in\Vee^*: [Q]^*\subseteq [T]^* \}, \subseteq)$ and $(\Vee^*, \subseteq)$ are isomorphic and separative. So,   one obtains a proof  analogous to the proof of the lemma  2.6 in   \cite{bps}. 
\end{proof}

\begin{theorem} \label{tco3} $ \kappa(\prod X_i)$ is a regular  uncountable cardinal number.
\end{theorem}
\begin{proof} Using Theorem \ref{*fuzja}, one can proceed analogous as in  \cite{bps} (corollaries   2.8 and 2.9) or  as in  \cite{bs} (Proposition 3.3). 
\end{proof}

\begin{theorem} \label{12} 
There exists a $v$-matrix $\Hee =\{ \Pee_\alpha: \alpha < \kappa(\prod X_i) \} $ such that if $\alpha<\beta<\kappa(\prod X_i)$, then $\Pee_\beta\prec\Pee_\alpha$. Moreover, for any  $[T]^*\in\Vee^*$ there exists  $[P]^*\in\bigcup \Hee$ where $[P]^*\subseteq [T]^*$. 
\end{theorem}
\begin{proof} A proof is completely analogous to the proof of Lemma 2.11 in  \cite{bps}, or to the proof of Theorem 3.4  in \cite{bs}.
\end{proof}
\section{Cardinal invariants}
In this part results hold for  the $\sigma$-ideal $d^0(\Vee^*)$.  One can check $d^0(\Vee^*)\subseteq d^0(\Vee)$ similarly as in the proof of Corollary \ref{rownosc}. We do not know whether $d^0(\Vee^*) = d^0(\Vee)$. To obtain results for the ideal $d^0(\Vee)$ we have to     assume that sets $X_i$  are finite.

\begin{lemma}\label{CCC} If  $\Pee$ is a  $v$-partition, then the complement of the union   $\bigcup \Pee$ belongs to  $d^0(\Vee^*)$.
\end{lemma} 
\begin{proof}
 If $[T]^*\in\Vee^*$, then take $[P]^*\in\Pee$ such that $[P]^*\cap [T]^*\in\Vee^*$.  Since $[P]^*\cap [T]^*\subseteq \bigcup\Pee$ and $[P]^*\cap [T]^*\subseteq [T]^*$, we are done. 
  \end{proof}

\begin{lemma}\label{DDD} 
If   $S \in d^0(\Vee^*)$, then there exists a $v$-partition $\Pee$ such that $\bigcup \Pee\cap S = \emptyset$. 
\end{lemma} 
\begin{proof}
For each $[T]^*\in\Vee^*$ fix  $[P]^*\in\Vee^*$ such that $[P]^*\subseteq [T]^*$ and $[P]^*\cap S =\emptyset $. Any  $v$-partition consisting of just fixed    $[P]^*$ is a desired one.
 \end{proof}

Observe that, if  $\Pee$ is a  $v$-partition and $S\subseteq \prod X_i$ is a selector of $\Pee$, i.e. $S \cap  [P]^*$ has exactly one point for each  $[P]^*\in \Pee$, then $S \in d^0(\Vee)$. Indeed, for any $[T]^*\in \Vee^*$, there exist $[P]^*\in\Pee$ such that $[T]^*\cap [P]^*\in\Vee^*$. Then $[T]^*\cap [P]^*\cap S$ has no more than one point. By Lemma \ref{continuum}, there exists $[Q]^*\in\Vee^*$ which is disjoint with $S$ and such that 
$[Q]^* \subseteq [T]^*$.

\begin{theorem} \label{15} $\kappa(\prod X_i) = add(d^0(\Vee^*))$.
\end{theorem}  
\begin{proof} 
Let  $\Fee\subseteq d^0(\Vee^*)$ and $|\Fee| < \kappa(\prod X_i)$. By Lemma \ref{DDD}, for each $W \in \Fee$ choose   a $v$-partition $\Pee_W$ such that $\bigcup \Pee_W \cap W =\emptyset$. By Lemma \ref{EEE} there exists a $v$-partition $\Pee$ which refines each $v$-partition $\Pee_W$ for $W \in \Fee$. The  set $\prod X_i\setminus\bigcup \Pee$ is element of $d^0(\Vee^*)$ and contains $\bigcup\Fee$. Thus, we have showed that $add(d^0(\Vee^*))\geq \kappa(\prod X_i)$.

Let $\{\Pee_\alpha:  \alpha < \kappa(\prod X_i)\}$ be a $v$-matrix like in the Theorem \ref{12}. We will construct inductively $v$-matrix $\{\Qee_\alpha: \alpha<\kappa(\prod X_i)\}$ such that for any $\alpha<\kappa(\prod X_i)$, 
$\Qee_\alpha\prec\Pee_\alpha$   and
 if $V\in\Qee_\alpha$, then $V\setminus\bigcup\Qee_{\alpha+1}\neq\emptyset$.
 Suppose, that we have already defined a $v$-partition $\Qee_\beta$ and assume that $\alpha=\beta+1<\kappa(\prod X_i)$.  Fix a selector $S$ of $v$-partition $\Qee_\beta$. By Lemma \ref{DDD} there exists a $v$-partition $\Qee$ such that $\bigcup\Qee\cap S=\emptyset$. Let $\Qee_\alpha$ be any  $v$-partition which refines the $v$-partition $\Qee$ and  the $v$-partition $\Pee_\alpha$. The remaining inductive steps are obvious. 
For any $T\in\Vee$ the intersection $[T]^*\cap\bigcup\{\prod X_i\setminus\bigcup\Qee_\alpha: \alpha<\kappa(\prod X_i)\}$ is nonempty. Indeed, for any $T\in\Vee$ choose $\alpha<\kappa(\prod X_i)$ and $[P]^*\in\Qee_\alpha$  such that $[P]^*\subseteq [T]^*$. Thus $\emptyset\neq [P]^*\setminus\bigcup\Qee_{\alpha+1}\subseteq \bigcup\{\prod X_i\setminus\bigcup\Qee_\alpha: \alpha<\kappa(\prod X_i)\}.$ Therefore, the family $\{\prod X_i\setminus\bigcup\Qee_\alpha: \alpha<\kappa(\prod X_i)\}$ witnesses that $add(d^0(\Vee^*))\leq\kappa(\prod X_i)$.
\end{proof}

\begin{theorem}\label{FEE} If all sets $X_i$ are countable, then $\omega_1\leq \kappa(\prod X_i) = add(d^0(\Vee^*)) \leq cov(d^0(\Vee^*)) \leq  \cf (\frak c ).$
\end{theorem} \begin{proof}

By Lemma \ref{continuum} each set $S\subseteq \prod X_i$ of the cardinality less than continuum belongs to $d^0(\Vee)$.
Choose sets $S_\alpha\subseteq \prod X_i\mbox{ for } \alpha < \cf(\frak c)$ such that $\prod X_i = \bigcup\{S_\alpha: \alpha < \cf(\frak c)\}$ and  $|S_\alpha| <\frak c.$ Thus $ cov(d^0(\Vee)) \leq \cf(\frak c)$.
  \end{proof}

If $\Pee$ is a  $v$-partition, then one can choose  subsets $N_C \subset C \in \Pee$  of the cardinality less than $\frak{c}$ such that  sets $A\setminus N_{A}$ and $B\setminus N_{B}$ are disjoint for any distinct members $A$, $B$ of $\Pee$.  One can do this by the induction using the fact that the intersection of any distinct members  of $\Pee$ is countable.
But, if  $\{\Pee_\alpha :\alpha<\kappa(\prod X_i)\}$ is a $v$-matrix like in Theorem \ref{12}, then put $M_C=\bigcup\{N_{C_\beta}: \beta\leq \alpha\}$ whenever  $C\in\Pee_\alpha$ and  $C\subseteq C_\beta\in \Pee_\beta$.
If all sets $X_i$ are countable, then sets $M_C$ have cardinalities less than the continuum, by Theorem \ref{FEE}.  Under  such assumptions, the family   $\{C\setminus M_C: C\in\bigcup\{\Pee_\alpha: \alpha<\kappa(\prod X_i)\}\}$
is called \textit{base matrix}.
Any base matrix consists of sets which are  either  disjoint or one is contained in the other. The topology on $\prod X_i$ generated by a base matrix is called \textit{matrix topology}. So, the family 
 \begin{center}
$\{\prod X_i\}\cup\{C\setminus M_C: C\in\bigcup\{\Pee_\alpha: \alpha<\kappa(\prod X_i)\}\}$ 
\end{center}
is a base for the matrix topology. 

\begin{lemma}\label{top} Suppose sets $X_i$ are countable.
A subset $Y\subset \prod X_i$ is nowhere dense with respect to a matrix topology if and only, if $Y\in d^0(\Vee^*)$.
\end{lemma}

\begin{proof} Fix a base matrix $$\Bee =\{C\setminus M_C: C\in\bigcup\{\Pee_\alpha: \alpha<\kappa(\prod X_i)\}\}.$$ At first, let  $Y\subset \prod X_i$ be a nowhere dense with respect to the matrix topology. Fix a set $V\in\Vee^*$ and choose  $W\subseteq V$ such that $W\in\Bee$ and $W\cap Y= \emptyset$.  Since  $W= C\setminus M_C$, where $C\in \Vee^*$ and  $|M_{C}|<\frak{c}$, the set $W$ contains some $D\in \Vee^*$. Any such $D$ witnesses that $Y\in d^0(\Vee^*)$. On the other hand, let  $U$ be a non-empty open set. Choose  $V\in\bigcup\{\Pee_\alpha :\alpha<\kappa(\prod X_i)\}$ such that $V\setminus M_{V}\subseteq U$. If  $Y\in d^0(\Vee^*)$, then there exists  $W\subseteq V$ such that $W\cap Y=\emptyset$ and $W\in\bigcup\{\Pee_\alpha :\alpha<\kappa(\prod X_i)\}$. Since   $W\setminus M_{W}\subseteq V\setminus M_{V}\subseteq U$ we conclude that  $Y \cap W \setminus M_{W}=\emptyset$.
\end{proof} 

One can find the next theorem  in the paper by W. Kulpa and  A. Szymański \cite{ks}. It is presented with a proof in  \cite{bps}.

\textbf{Theorem}. 
\textit{Let $\Wee $ be a collection of families consisting of open subsets of a topological space $Y$. Suppose that:  
	 $\bigcup\Wee $ is a  $\pi$-base;
		 any family in $ \Wee$  consists of pairwise disjoint sets;
	 $|\Wee |<\tau $,  where $\tau$ is a regular cardinal number; 
	each set belonging to $\bigcup\Wee$ contains $\tau$ many pairwise disjoint open sets.
Then there exists an increasing family of nowhere dense subsets  $\{Y_\alpha: \alpha<\tau\}$ such that  $\bigcup\{Y_\alpha: \alpha<\tau\}=Y. $}

Thus, we can estimate $cov (d^0(\Vee))$, more accurate than these in \cite{kpw}. 

\begin{theorem}\label{final} If sets $X_i$ are countable, then
$$add(d^0(\Vee^*))\leq cov(d^0(\Vee^*))\leq add(d^0(\Vee^*))^+.$$
\end{theorem}
\begin{proof}
Obviously, if  $\kappa(\prod X_i)=\frak{c}$, then $\frak{c}=add(d^0(\Vee^*))= cov(d^0(\Vee^*))$. Suppose that  $\kappa(\prod X_i)<\frak{c}$, them the  above theorem, i.e. the theorem by Kulpy and Szymański, works. Let $\Wee$ be a base matrix. Put   $\tau = \kappa(\prod X_i)^+$. Then  $\bigcup\Wee $ is a  $\pi$-base for the  topology generated by itself on $\prod X_i$. Each  $V\in \bigcup\Wee$ contains  $\frak{c}$-many elements of $\bigcup\Wee$.  By Corollary  \ref{rownosc} and Lemma \ref{top} and the theorem by Kulpy and Szymański we obtain   $cov(d^0(\Vee^*))\leq \kappa(\prod X_i)^+$. We are done, since Theorem  \ref{15}.
\end{proof}

Thus, if sets $X_i$ are countable, then $cov(d^0(\Vee))\leq cov(d^0(\Vee^*)).$ If sets $X_i$ are finite, then $$add(d^0(\Vee))\leq cov(d^0(\Vee))\leq add(d^0(\Vee))^+.$$

\end{document}